\documentclass[11pt]{article}

\usepackage{amssymb,amsmath,amsfonts,amsthm}
\usepackage{latexsym}
\usepackage{graphics}
\usepackage{indentfirst}

\usepackage{enumerate}
\usepackage{amssymb}
\usepackage{tikz}\usetikzlibrary{calc}
\usepackage{hyperref}

\usepackage{comment}

\allowdisplaybreaks

\newenvironment{graph}[1][scale=1]{
\begin{tikzpicture}[#1]
\tikzstyle{vertex}=[circle, draw, fill, inner sep=0pt, minimum size=4pt]%
\tikzstyle{every path}=[line width=1pt]%
\tikzstyle{G}=[dashed]%
\tikzstyle{F}=[solid]
}{\end{tikzpicture}}

\newtheorem{theorem}{Theorem}
\newtheorem{lemma}[theorem]{Lemma}

\newtheorem{cor}[theorem]{Corollary}
\newtheorem{question}[theorem]{Question}
\newtheorem*{definition}{Definition}

\renewcommand{\emptyset}{\varnothing}

\renewcommand\footnotemark{}

\begin{document}

\title{A Generalization of the Graph Packing Theorems of Sauer-Spencer and Brandt}

\author{Hemanshu Kaul \and Benjamin Reiniger}
\thanks{Department of Applied Mathematics, Illinois Institute of Technology, Chicago, IL 60616. Email: \texttt{kaul@iit.edu, ben.reiniger-math@yahoo.com}.}

\date{}

\maketitle

\begin{abstract}
We prove a common generalization of the celebrated Sauer-Spencer packing theorem and a theorem of Brandt concerning finding a copy of a tree inside a graph. This proof leads to the characterization of the extremal graphs in the case of Brandt's theorem: If $G$ is a graph and $F$ is a forest, both on $n$ vertices, and $3\Delta(G)+\ell^*(F)\leq n$, then $G$ and $F$ pack unless $n$ is even, $G=\frac{n}{2}K_2$ and $F=K_{1,n-1}$; where $\ell^*(F)$ is the difference between the number of leaves and twice the number of nontrivial components of $F$.




\end{abstract}

\section{Introduction}

Given two graphs $G$ and $H$ both on $n$ vertices, we say that $G$ and $H$ \emph{pack} if there is a bijection $f: V(G)\to V(H)$ such that for every $uv\in E(G)$, $f(u)f(v)\notin E(H)$; in other words, edge-disjoint copies of $G$ and $H$ can be found in $K_n$, or equivalently, $G$ is
isomorphic to a subgraph of the complement of  $H$. This concept leads to a natural generalization of a number of problems in extremal graph theory, such as existence of a fixed subgraph, equitable colorings, and Turan-type problems. The study of packing of graphs was started in the 1970s by Bollob\'{a}s and Eldridge~\cite{BE}, Sauer and Spencer~\cite{SS}, and Catlin~\cite{C1}. See the surveys by Kierstead et al.~\cite{KKY2009}, Wozniak~\cite{W}, and Yap~\cite{Y} for later developments in this field. In the following, we will use $\Delta(G)$ ($\delta(G)$) to denote the maximum (resp., minimum) degree of a graph $G$.

The major conjecture in graph packing is that of Bollob\'as and Eldridge~\cite{BE}, and independently by Catlin~\cite{C2}, from 1978, that $(\Delta(G)+1)(\Delta(H)+1)\leq n+1$ is sufficient for $G$ and $H$ to pack. Some partial results are known, e.g.~\cite{Eaton, KKY, AB, AF, Csaba, CSS, BK2016}. 

In 1978, Sauer and Spencer proved the following celebrated result.
\begin{theorem}[Sauer, Spencer~\cite{SS}]
Let $G,H$ be graphs on $n$ vertices such that $2\Delta(G)\Delta(H)<n$.  Then $G$ and $H$ pack.
\end{theorem}
Kaul and Kostochka~\cite{KK} strengthened the result by characterizing the extremal graphs: if $2\Delta(G)\Delta(H)=n$ and $G$ and $H$ fail to pack, then $n$ is even, one of the graphs is $\frac n2 K_2$, and the other is either $K_{n/2,n/2}$ (with $n/2$ odd) or contains $K_{n/2+1}$.

Let $\ell(F)$ denote the number of leaves in a forest $F$.  In 1994, Brandt~\cite{Brandt} proved that if $G$ is a graph and $T$ is a tree, both on $n$ vertices, and $\ell(T)\leq 3\delta(G)-2n+4$, then $G$ contains a copy of $T$.
This can be rephrased in terms of packing.
\begin{theorem}[Brandt~\cite{Brandt}]
If $G$ is a graph and $T$ is a tree, both on $n$ vertices, and
\[ 3\Delta(G)+\ell(T)-2<n, \]
then $G$ and $T$ pack.
\end{theorem}
We need a generalization of this theorem to a forest $F$, which is straightforward and motivates the following definition.

\begin{definition}
The \emph{excess leaves} of a forest $F$, denoted $\ell^*(F)$, is $\sum_{v\in V(F)} \max\{d(v)-2, 0\}$.
\end{definition}
Note that linear forests are precisely the forests with zero excess leaves.  We also have that $\ell^*(F)$ equals the number of leaves of $F$ minus twice the number of nontrivial components of $F$ (those having at least two vertices), and that for a tree $T$, $\ell^*(T)=\ell(T)-2$.\footnote{For these, consider the sum $\sum_{i\geq0}(i-2)n_i$, where $n_i$ is the number of vertices with degree $i$.}

\begin{cor}\label{cor:brandt-forest}
If $G$ is a graph and $F$ is a forest, both on $n$ vertices, and $3\Delta(G)+\ell^*(T)<n$, then $G$ and $F$ pack.
\end{cor}
\begin{proof}
Iteratively add edges joining leaves of distinct nontrivial components of $F$; each such addition does not change $\ell^*$.  When there is only one nontrivial component left, iteratively add edges from any leaf to the remaining (isolated) vertices; again $\ell^*$ is preserved.  Now we have a tree, for which $\ell^*=\ell-2$.  Brandt's theorem now applies, so that $G$ and the new tree pack, and deleting the added edges gives a packing of $G$ with $F$.
\end{proof}

Corollary~\ref{cor:brandt-forest} is sharp when  $n$ is even, with $G=\frac n2 K_2$ and $F=K_{1,n-1}$.  We will prove that this is the only pair of extremal graphs, strengthening Brandt's result as follows.
\begin{theorem}\label{thm:onlysharpness}
If $G$ is a graph and $F$ is a forest, both on $n$ vertices, and
\[ 3\Delta(G)+\ell^*(F)\leq n,\]
then $G$ and $F$ pack unless $n$ is even, $G=\frac{n}{2}K_2$, and $F=K_{1,n-1}$.
\end{theorem}

To accomplish this, we will first prove the following theorem, which generalizes both the Sauer-Spencer and Brandt packing theorems.
\begin{theorem}\label{thm:degSS}
Let $G$ be a graph and $H$ a $c$-degenerate graph, both on $n$ vertices.  Let $d_1^{(G)}\geq d_2^{(G)}\geq \dotsb \geq d_n^{(G)}$ be the degree sequence of $G$, and similarly for $H$.  If
\[ \sum_{i=1}^{\Delta(G)} d_i^{(H)} + \sum_{j=1}^{c} d_j^{(G)} < n, \]
then $G$ and $H$ pack.
\end{theorem}

This strengthens Sauer-Spencer, since
$c\leq \Delta(H)$.

This also strengthens Brandt's theorem: if $H$ is a tree, then $c=1$, so the second summation is just $\Delta(G)$.  For the first summation,
\[  \sum_{i=1}^{\Delta(G)}d_i^{(H)} = 2\Delta(G)+\sum_{i=1}^{\Delta(G)} \big(d_i^{(H)}-2\big) \leq 2\Delta(G)+\ell(H)-2. \]

It is easy to construct examples of graphs $G$ and $H$ for which conditions in the Sauer-Spencer theorem, Brandt's theorem, or even the  Bollob\'as, Eldridge, and Catlin conjecture are not true, but Theorem~\ref{thm:degSS} does apply.

The proof of Theorem~\ref{thm:degSS} generally follows that of Sauer-Spencer.  In the special setting of Brandt's theorem, the proof can be analyzed more closely to show that the only sharpness example is the one mentioned above.

Theorem~\ref{thm:degSS} is sharp itself, with several sharpness examples.
It retains all the Sauer-Spencer sharpness examples (with $n$ even) mentioned earlier:
\begin{itemize}
\item $H=\frac n2 K_2$ and $G\supseteq K_{n/2+1}$
\item $H=\frac n2 K_2$ and $G=K_{n/2,n/2}$, with $n/2$ odd
\item $H\supseteq K_{n/2+1}$ and $G=\frac n2 K_2$
\item $H=K_{n/2,n/2}$ and $G=\frac n2 K_2$, with $n/2$ odd
\end{itemize}
And it has an additional family of sharpness examples:
\begin{itemize}
\item $H=K_{s,n-s}$ and $G=\frac n2 K_2$, with $n$ even and $s$ odd\\
(in particular, $H=K_{1,n-1}$ and $G=\frac n2 K_2$)
\end{itemize}
We do not know whether these are all the sharpness examples, even if we restrict to the case that $H$ is a forest. 

\begin{question}\label{ques:extremaldegSS}
What are the extremal graphs for Theorem~\ref{thm:degSS}? Do the above listed families of graphs include all the extremal graphs for Theorem~\ref{thm:degSS} when $H$ is a forest? 
\end{question}

Note that Theorem~\ref{thm:onlysharpness} shows that the only extremal graphs for that theorem have $n$ even and $\Delta
(G)=1$. So, it is natural to ask:
\begin{question}\label{ques:extension}
By Theorem~\ref{thm:onlysharpness},  $3\Delta(G)+\ell^*(F)< n+1$ is a sufficient condition for packing of a graph $G$ and a forest $F$ on $n$ vertices when $n$ is odd or $\Delta(G) \ge 2$. Is this statement sharp? If yes, what are all its sharpness examples?
\end{question}

Degeneracy versions of the Sauer-Spencer packing theorem have been studied before, in~\cite{BKN} and~\cite{KKgame}. 
If we think of the condition in Sauer-Spencer as the sum of two terms: $\Delta(G)\Delta(H) + \Delta(H)\Delta(G) <n$, then Theorem~\ref{thm:degSS} can be thought of as replacing $\Delta(H)$ by the degeneracy $c(H)$ in one the terms (in addition to other degree sequence related improvements). The result in~\cite{BKN} replaces $\Delta(G)$ by $c(G)$ in one term and $\Delta(H)$ by $\log\Delta(H)$ in the other. In~\cite{KKgame}, $\Delta(G)$ is replaced by $(\operatorname{gcol}(G)-1)$ in one term and  $\Delta(H)$ by $(\operatorname{gcol}(H)-1)$ in the other, where $\operatorname{gcol}$ denotes the \emph{game coloring number} and $\operatorname{gcol}(G)-1$ lies in between the degeneracy and the maximum degree (see~\cite{KKgame} for precise definition and details). It is natural to ask for improvements or extensions of Theorem~\ref{thm:degSS} by considering degree-sum conditions that interplay between maximum degree, degeneracy, and game coloring number. For example, does $\sum_{i=1}^{gcol(G)-1} d_i^{(H)} + \sum_{j=1}^{gcol(H)-1} d_j^{(G)} < n$ suffice for a packing of $G$ and $H$ under the set-up of Theorem~\ref{thm:degSS}? Or, does $c_1 \sum_{i=1}^{\lceil\log\Delta(G)\rceil} d_i^{(H)} <n$ and  $c_2\sum_{j=1}^{c(H)} d_j^{(G)} < n$ for some fixed constants $c_1$ and $c_2$ suffice?

\section{Proofs}
Throughout, we think of a bijective mapping $f: V(G)\to V(H)$ as the multigraph with vertices $V(G)$ and edges labelled by ``$G$'' or ``$H$''.
We speak of $H$- and $G$-edges, $H$- and $G$-neighbors of vertices, and $H$-cliques, $H$-independent sets, etc.
A \emph{link} is a copy of $P_3$ with one $H$-edge and one $G$-edge, and a $uv$-link is a link with endpoints $u$ and $v$; a \emph{$GH$-link from $u$ to $v$} is a link with endpoints $u,v$ whose edge incident to $u$ is from $G$; similarly we have $HG$-links.  From a given mapping $f$, a \emph{$uv$-swap} results in a new mapping $f'$ with $f'(u)=f(v)$, $f'(v)=f(u)$, and $f'=f$ otherwise.  A \emph{quasipacking} of $G$ with $H$ is a mapping $f$ whose multigraph is simple except for a single pair of vertices joined by both an $H$-edge and a $G$-edge; this pair is called the \emph{conflicting edge} of the quasipacking.

Consider a pair of graphs $(G,H)$, with $H$ being $c$-degenerate, each on $n$ vertices, that do not pack; furthermore assume that $H$ is edge-minimal with this property.
Thus for any edge $e$ in $H$, $G$ and $H-e$ pack, and so there is a quasipacking of $H$ and $G$ with conflicting edge $e$.

Let $u'$ be a vertex of minimum positive degree in $H$, let $x'\in N_H(u)$, and consider a quasipacking $f$ of $G$ with $H$ with conflicting edge $u'x'$.  Let $u=f^{-1}(u')$ and $x=f^{-1}(x')$.
We will now consider the set of links from $u$ to each vertex.

Consider a $y\in V(G)\setminus\{u,x\}$.  Perform a $uy$-swap: since $G$ and $H$ do not pack, there must be some conflicting edge, and such a conflict must involve an $H$-edge incident to either $u$ or $y$; together with the conflicting $G$-edge, we have a $uy$-link in the original quasipacking.
There are two links from $u$ to itself, using the parallel edges $ux$ in each order.
Thus there are at least $n$ links from $u$ in the original quasipacking $f$.

The number of $GH$-links from $u$ is at most $\sum_{y\in N_G(u)} \deg_H(f(y))$.  The number of $HG$-links from $u$ is at most $\sum_{z'\in N_H(u')} \deg_G(f^{-1}(z'))$.  Hence we have
\begin{equation}\label{eq:GH-links}
n \leq \text{\# links from $u$} \leq
  \sum_{y\in N_G(u)} \deg_H(f(y)) + \sum_{z'\in N_H(u')} \deg_G(f^{-1}(z'))
   \leq \sum_{i=1}^{\Delta(G)} d^{(H)}_i + \sum_{j=1}^{c} d^{(G)}_j.
\end{equation}
This establishes Theorem~\ref{thm:degSS}.\\

To prove Theorem~\ref{thm:onlysharpness}, suppose additionally that $H$ is a forest, henceforth called $F$, and that $3\Delta(G)+\ell^*(F)=n$.  (So, we still assume that $G$ and $F$ do not pack, and that $F$ is edge-minimal with this property.)

If $\Delta(G)=1$, then it is easy to show that $n$ is even, $G=\frac{n}{2} K_2$, and $F=K_{1,n-1}$.  (In fact, such a $G$ will pack with any bipartite graph that is not complete bipartite.)  So we henceforth assume that $\Delta(G)>1$, and seek a contradiction.

\begin{lemma}\label{lemma}
For any leaf $u'$ of $F$ and $x'$ its neighbor, and a quasipacking $f$ of $G$ with $F$ with $f(u)=u'$ and $f(x)=x'$ and conflicting edge $ux$, we have the following.
\begin{enumerate}[1.]
\item\label{property:links} For every $y\in V(G)\setminus\{u,x\}$, there is a unique link from $u$ to $y$; there is no link from $u$ to $x$; and there are two links from $u$ to itself.
\item\label{property:edegs} $\deg_G(x)=\deg_G(u)=\Delta(G)$.
\item\label{property:Flargedegs} For every $w\in N_G(u)$, $\deg_F(f(w))\geq2$.
\item\label{property:Fsmalldegs} For every $w\notin N_G(u)$, $\deg_F(f(w))\leq 2$.
\end{enumerate}
\end{lemma}
\begin{proof}
Note that we now have $\deg_F(u')=1$, so $\sum_{z'\in N_F(u')}\deg_G(f^{-1}(z')) = \deg_G(x)$.
In this case we can expand on (\ref{eq:GH-links}):
\begin{align}
 n \leq \text{\# links from $u$} &\leq \sum_{y\in N_G(u)} \deg_F(f(y)) + \deg_G(x) \\
 &\leq \sum_{y\in N_G(u)} \big(\deg_F(f(y))-2\,\big) + 2\Delta(G) + \Delta(G) \\
 &\leq \sum_{y\in N_G(u)} \max\{\deg_F(f(y)) -2 , 0\} + 3\Delta(G) \\
 &\leq \sum_{i=1}^n \max\{ d^{(F)}_i -2, 0\} + 3\Delta(G)
             = \ell^*(F) + 3\Delta(G) = n,
\end{align}
so we have equality throughout.
Conclusion $i$ follows from having equality in line ($i+1$) above, for $i\in[4]$.
\end{proof}

For a vertex $v$ in a graph $H$, we write $N_H[v]$ for the closed neighborhood, i.e. $N_H(v)\cup\{v\}$.  For a set $S$ of vertices, $N_H(S)=\bigcup_{v\in S} N_H(v) -S$.

\begin{lemma}\label{lemma-parts}
For any leaf $u'$ of $F$ and $x'$ its neighbor, and a quasipacking $f$ of $G$ with $F$ with $f(u)=u'$ and $f(x)=x'$ and conflicting edge $ux$, we have the following.
\begin{enumerate}[1.]
\item \label{lemma-partition}
$N_G[u]=N_G[x]$.
\item \label{lemma-clique}
Let $Q=N_G[u]$.  Then $G[Q]$ is a clique component.
\end{enumerate}
\end{lemma}

\begin{proof}
Proof of part \ref{lemma-partition}.

Let $A=N_G(u)-N_G[x]$, $B=N_G(u)\cap N_G(x)$, $C=N_G(x)-N_G[u]$.  Also, let $N_A=N_F(f(A))$, $N_B=N_F(f(B))$, $N_C=N_F(f(C))$, and $N_x=N_F(x')$.

We will show that $A=N_A=N_C=C=\emptyset$.

\begin{figure}[h]
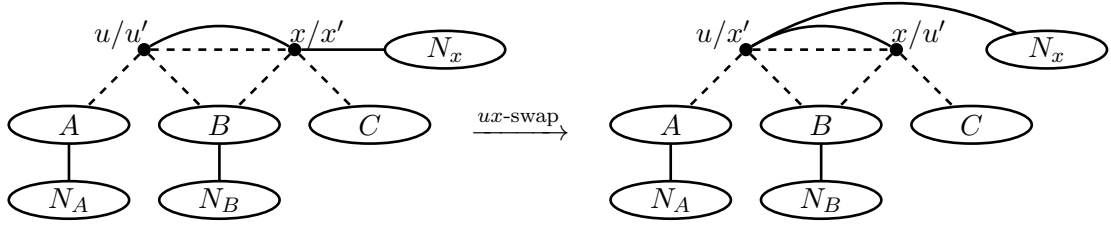

\centering
\begin{graph}
\node[vertex] (u) at (-1,0) {};
\node[vertex] (x) at (1,0) {};
\node (A) at (-2,-1) {};   \node (N_A) at (-2, -2) {};
\node (B) at (0,-1) {};    \node (N_B) at (0,-2) {};
\node (C) at (2,-1) {};    \node (N_x) at (3,0) {};
\draw[G] (u)--(x);  \draw[F] (u) to [bend left] (x);
\draw[G] (u)--(A);  \draw[G] (u)--(B);  \draw[G] (x)--(B);  \draw[G] (x)--(C);
\draw[F] (x)--(N_x);
\draw[F] (A)--(N_A);
\draw[F] (B)--(N_B);
\foreach \i in {A,N_A,B,N_B,C,N_x}
{
   \draw[fill=white] (\i) ellipse [x radius=0.8, y radius=0.25];
   \node at (\i) {$\i$};
}
\node at ($(u)+(-0.3,0.2)$) {$u$/$u'$};
\node at ($(x)+(0.3,0.2)$) {$x$/$x'$};
\begin{scope}[shift={(4,0)}]
\node at (0,-1) {$\xrightarrow{\text{$ux$-swap}}$};
\end{scope}
\begin{scope}[shift={(8,0)}]
\node[vertex] (u) at (-1,0) {};
\node[vertex] (x) at (1,0) {};
\node (A) at (-2,-1) {};   \node (N_A) at (-2, -2) {};
\node (B) at (0,-1) {};    \node (N_B) at (0,-2) {};
\node (C) at (2,-1) {};    \node (N_x) at (3,0) {};
\draw[G] (u)--(x);  \draw[F] (u) to [bend left] (x);
\draw[G] (u)--(A);  \draw[G] (u)--(B);  \draw[G] (x)--(B);  \draw[G] (x)--(C);
\draw[F] (u) to [bend left] (N_x);
\draw[F] (A)--(N_A);
\draw[F] (B)--(N_B);
\foreach \i in {A,N_A,B,N_B,C,N_x}
{
   \draw[fill=white] (\i) ellipse [x radius=0.8, y radius=0.25];
   \node at (\i) {$\i$};
}
\node at ($(u)+(-0.3,0.2)$) {$u$/$x'$};
\node at ($(x)+(0.3,0.2)$) {$x$/$u'$};
\end{scope}
\end{graph}
\caption{Left: the quasipacking $f$, with $G$-edges dashed and $F$-edges solid.  Right: the result after the $ux$-swap.}
\label{figure:init}
\end{figure}

Note that $B\cup C\cup \{u\}$ is precisely the set of vertices with an $FG$-link from $u$.
By Lemma~\ref{lemma}(\ref{property:links}), there are no $F$-edges from $A$ to $x$, else $x$ would have a $GF$-link; and there are no $F$-edges from $A$ to $B\cup C\cup N_x$, else such an endpoint in $B\cup C\cup N_x$ would have two links.
So for each vertex of $A$ to have exactly one link, $F[f(A)]$ must be a perfect matching.
Furthermore, the $F$-edges incident to $A$ only have endpoints in $A\cup N_A$.
Each vertex of $N_A$ must have exactly one $F$-edge from $A$ (to have one link).  And by Lemma~\ref{lemma}(\ref{property:Flargedegs}), each vertex of $A$ has at least one $F$-neighbor in $N_A$.  Note that we thus have $|N_A|\geq|A|$.
The vertices of $N_A, N_B, N_{x}$ all have $GF$-links by definition, and to have exactly one, these sets must be disjoint.  Thus we have that $\{u,x\},A,B,C,N_A,N_B,N_{x}$ is a partition of $V(G)$.  See the left side of Figure~\ref{figure:init}.

Now perform a $ux$-swap.  In Figure~\ref{figure:init}, we visualize with $V(G)$ fixed, so just the $F$-edges adjacent to $u'$ and $x'$ move; roughly speaking, we just interchange the roles of $u$ and $x$ and those of $A$ and $C$.
The result is again a quasipacking with $ux$ the only conflicting edge.  The $F$-neighbors of $u$ are precisely $N_{x}$.
Repeating the arguments of the last paragraph, for each vertex of $N_C$ to have exactly one link, we must have $|N_C|\geq|C|$.
Suppose that $A\neq\emptyset$.  Then since $|N_A|\geq|A|$, $N_A\neq\emptyset$ as well.
Now, the only possible links (from $x$) to vertices in $N_A$ are $GF$-links through $C$; hence $N_C=N_A$, and the $F$-edges incident to $C$ have endpoints in $C\cup N_A$.
Furthermore, since $N_C=N_A\neq\emptyset$, $C\neq\emptyset$ as well; so each vertex of $N_A$ has $F$-degree at least 2 (one edge from $A$ and one from $C$).
But also, by Lemma~\ref{lemma}(\ref{property:Flargedegs}) applied to the original and also this new quasipacking, every vertex of $A$ and $C$ has $F$-degree at least 2, with $F$-edges entirely in $A\cup C\cup N_A$.
So $F[f(A\cup C\cup N_A)]$ has minimum degree at least 2, contradicting that it is a forest, unless $A=C=N_A=\emptyset$.  Note that this implies that $N_G[u]=N_G[x]=\{u,x\}\cup B$.

Proof of part~\ref{lemma-clique}.

This time perform a $uy$-swap for some vertex $y\in Q\setminus\{u,x\}$ to get $\tilde{f}$.  The result is again a quasipacking with $yx$ the only conflicting edge, with $\tilde{f}(y)=u'$.  By part~\ref{lemma-partition}, $N_G[y]=N_G[x]=Q$.  Since this holds for every $y\in Q\setminus\{u,x\}$, we have that $Q$ is a clique; and since $\deg_G(x)=\Delta(G)$ by Lemma~\ref{lemma}(\ref{property:edegs}), $G[Q]$ is a clique component of $G$.
\end{proof}

Let $u'$ be a leaf in $F$, and let $x'$ be its neighbor.  Consider a quasipacking $f$ of $G$ with $F$ with $f(u)=u'$ and $f(x)=x'$ and conflicting edge $ux$.  (Such exists by the extremal choice of $F$, as in the proof of Theorem~\ref{thm:degSS}.)

Let $G[Q]$ be the clique component of $G$ given in Lemma~\ref{lemma-parts}(\ref{lemma-clique}).
Let $z$ be a vertex of $Q$ with smallest $F$-degree larger than 1 (such a choice is possible, as $\deg_F(x')\geq2$ by Lemma~\ref{lemma}(\ref{property:Flargedegs})), and let $z'=f(z)$.
Let $z_1, z_2 \in V(G)$ be two $F$-neighbors of $z$.

In $f$, $z_1$ and $z_2$ each have
exactly one $F$-edge into $Q$ and at most one other $F$-edge, by Lemma~\ref{lemma}(\ref{property:links},\ref{property:Fsmalldegs}).
So $z_1, z_2$ have no $F$-neighbors inside $Q$ except $z$.
From this and that $Q$ is a $G$-clique in the quasipacking,
the set $Q\cup\{z_1,z_2\}\setminus\{z\}$ is $F$-independent (whether $z=x$ or not) except perhaps the conflicting edge $ux$.
So $Q\cup\{z_1,z_2\}\setminus\{u,z\}$ is $F$-independent.
Let $X=f(Q\cup\{z_1,z_2\}\setminus\{u,z\})$.

Let $g: V(G)\to V(F)$ be a bijection such that $g(Q)=X$.
Since $G[Q]$ is a clique component and $X$ is independent, $g$ is a packing if and only if $g|_{G-Q}$ is a packing of $G-Q$ with $F-X$.

\smallskip
\noindent\textbf{Claim}:  $\deg_F(z')\geq4$.\\
Suppose to the contrary that $\deg_F(z')\leq 3$.  We have taken two of the neighbors of $z$ into $X$, so $\deg_{F-X}(z')\leq 1$.  And $z'$ is the only vertex of $F-X$ that may have degree larger than 2, by Lemma~\ref{lemma}(\ref{property:Fsmalldegs}).  That is, $F-X$ is a linear forest.  We have that
$ \Delta(G-Q)\leq \Delta(G)\leq \frac n3$, so
\begin{align*}
 \delta(\overline{G-Q})  &= |V(G-Q)|-1 - \Delta(G-Q) \\
  &= n-(\Delta(G)+1) -1 -\Delta(G-Q) \\
  &\geq n-\frac32\Delta(G)-\frac12\Delta(G)-2 \\
  &\geq \frac12n-\frac12\Delta(G)-2 \\
  &= \frac{1}{2}|V(G-Q)|,
\end{align*}
and so Dirac's condition for Hamiltonicity applies (\cite{Dirac}).  Since $\overline{G-Q}$ contains a Hamiltonian cycle, it also contains the linear forest $F-X$, i.e.~$F-X$ and $G-Q$ pack, a contradiction.  This completes the proof of the Claim.
\smallskip

This Claim, together with having $z'\notin X$ but its two neighbors $z_1, z_2\in X$, gives us the inequality
\begin{align*}
\ell^*(F-X) &= \sum_{v\in V(F-X)} \max\{\deg_{F-X}(v)-2, 0\} \\
 &\leq -2 + \sum_{v\in V(F-X)} \max\{\deg_{F}(v)-2, 0\} \\
 &= -2 + \sum_{v\in V(F)} \max\{\deg_{F}(v)-2, 0\} - \sum_{v\in X} \max\{\deg_{F}(v)-2, 0\} \\
 &= -2 + \ell^*(F) - \sum_{v\in X}\max\{\deg_F(v)-2, 0\}.
\end{align*}

From Lemma~\ref{lemma}(\ref{property:Flargedegs}), every vertex of $f(Q-u)$ has $F$-degree at least two; and since $z'$ was chosen to have smallest $F$-degree among the non-leaves of $f(Q)$, the Claim gives that they must in fact have degree at least four.
All these vertices except $z'$ are in $X$, so we have at least $\Delta(G)-1$ vertices of $X$ with degree at least 4.  Hence
$\displaystyle 2+\sum_{v\in X}\max\{\deg_F(v)-2, 0\}\geq 2\Delta(G)>\Delta(G)+1$, so

\begin{align*}
3\Delta(G-Q)+\ell^*(F-X)
&\leq 3\Delta(G)+\ell^*(F) - 2-\sum_{v\in X}\max\{\deg_F(v)-2, 0\} \\
&= n - 2-\sum_{v\in X}\max\{\deg_F(v)-2, 0\} \\
&<n-\Delta(G)-1 \\
&= |V(G-Q)|.
\end{align*}
Thus, by Theorem~\ref{thm:degSS}, $G-Q$ and $F-X$ pack, a contradiction.  This completes the proof of Theorem~\ref{thm:onlysharpness}.
\hfill\qedsymbol
\vspace*{0.5cm}

{\bf Acknowledgment.} The authors thank the anonymous referees for their helpful suggestions for improving the exposition. 

\small
\bibliographystyle{abbrv}
\bibliography{GTpacking}
\end{document}